\documentclass[10pt]{amsart}
\usepackage{amsmath,amssymb,amsthm,epsfig,enumerate,youngtabExt,setspace,a4wide}

\newtheorem{thm}{Theorem}
\newtheorem{lemma}[thm]{Lemma}
\newtheorem{prop}[thm]{Proposition}

\theoremstyle{definition}
   \newtheorem{remark}[thm]{Remark}
\theoremstyle{definition}
   \newtheorem{example}[thm]{Example}

\newcommand{\ch}{\operatorname{ch}}
\newcommand{\mS}{\mathcal{S}}
\newcommand{\Tr}{\operatorname{Tr}}

\newcommand{\B}{\mathcal B}
\newcommand{\N}{\mathbb N}

\newcommand{\End}{\operatorname{End}}
\newcommand{\sgn}{\operatorname{sgn}}
\newcommand{\mbar}{\underline\mu}
\newcommand{\mhat}{\hat\mu}

\title[Young module multiplicities]{Young module multiplicities and classifying the indecomposable Young permutation modules}
\keywords{Young module; indecomposable; permutation module; symmetric group}
\subjclass[2010]{20C30}
\author{Christopher C. Gill}
\address{Department of algebra, Charles University, Sokolovska 83, Praha 8, 186 75, Czech Republic}
\email{gill@maths.ox.ac.uk}
\begin{document}
\begin{abstract}
We study the multiplicities of Young modules as direct summands of permutation modules on cosets of Young subgroups.  Such multiplicities have become known as the $p$-Kostka numbers.  We classify the indecomposable Young permutation modules, and, applying the Brauer construction for $p$-permutation modules, we give some new reductions for $p$-Kostka numbers.  In particular we prove that $p$-Kostka numbers are preserved under multiplying partitions by $p$, and strengthen a known reduction corresponding to adding multiples of a $p$-power to the first row of a partition.  
\end{abstract}
\maketitle

The symmetric group permutation modules on the cosets of Young subgroups are known as the Young permutation modules.  They play a central role in the representation theory of the symmetric group, and also in Schur algebras, relating representations of symmetric groups to polynomial representations of general linear groups. The indecomposable modules occurring in direct sum decompositions of Young permutation modules were parametrized by James, and have become known as the Young modules.  In the semisimple case, the Young modules are the familiar Specht modules, and their multiplicities as direct summands of permutation modules (the Kostka numbers) have long been known by Young's rule.  However, when the group algebra of the symmetric group is not semisimple, the multiplicities remain undetermined.  It is well-known that the determination of these multiplicities is equivalent to determining the decomposition numbers of the symmetric groups and the Schur algebras. 

Let $\mS_r$ be the symmetric group on $r$ letters, and let $k$ be a field of characteristic $p$.  Both the Young permutation modules and the Young modules for $k\mS_r$ are indexed by the partitions of $r$, and denoted by $M^\lambda$ and $Y^\lambda$ respectively.   
In this paper we focus on the $p$-Kostka numbers.  That is, the multiplicities $[M^\lambda:Y^\mu]$ of a Young module $Y^\mu$ as a direct summand of $M^\lambda$ for partitions $\lambda,\mu$ of $r$. Whilst a full determination of the $p$-Kostka numbers is out of reach, there are several known reduction formulae.  The first appears to be Klyachko's multiplicity formula (see for example \cite[Theorem 4.6.3(ii)]{MartinSchurAlgebras} or \cite{grabmeier}).  Donkin \cite[3.6]{DonkinTilting} has given an algorithmic description, related closely to Klyachko's formula,  of when a Young module occurs as a direct summand of a Young permutation module.  Other reductions have been given by Henke and Koenig in \cite{HenkeKoenig} and also by Fang, Henke, and Koenig in \cite{HenkeKoenigFang}, using Schur algebras and Ringel duality.  In \cite{HenkeKoenigFang} it is shown that a certain complement construction on partitions preserves $p$-Kostka numbers, and as a consequence, row and column removal formulae are then deduced.  Henke has determined 
the $p$-Kostka numbers corresponding to $2$-part partitions in \cite{henke}.

The paper is split into two halves.  In the first half, we determine several new reduction formulae for $p$-Kostka numbers. The first of these shows that multiplication of partitions by $p$ preserves the $p$-Kostka numbers as follows:
\begin{thm}\label{thm:multbyp}
Let $\lambda,\mu\vdash r$. The following holds: \[[M^{p\lambda}:Y^{p\mu}]=[M^\lambda:Y^\mu].\]  
\end{thm}
Theorem \ref{thm:multbyp} has been used by the author to prove that the multiplicities of Young modules as direct summands of tensor products of Young modules are preserved under multiplying the indexing partitions by $p$ (\cite[Theorem 3.6]{Gill2012}).  These results add to the known instances of stability under multiplying partitions by $p$ in the literature.  Such results are surveyed in \cite{HemmerTwistSurvey}.  Furthermore, as a consequence of Theorem \ref{thm:multbyp} we show (Theorem \ref{thm:decomp}) that if one knows the decomposition numbers of the principal block of the Schur algebra $S(rp,rp)$ (those indexed by partitions with empty $p$-core), then there is an algorithm to determine the decomposition numbers of the Schur algebra $S(r,r)$.

 Theorem \ref{thm:multbyp}, and further reductions are proved in Section \ref{sec:red}.  The main method is to use the Brauer construction applied to $p$-permutation modules as developed by Brou\'{e}, and the description of the Brauer quotients of Young modules given by Erdmann.  With some analysis of the combinatorics of Young vertices we first prove Theorem \ref{thm:multbyp}, and then use this to determine a lower bound for certain $p$-Kostka numbers corresponding to adding $p$-power multiples of partitions (Theorem \ref{thm:plusalphaplusdelta}).  Whilst this in general is an inequality, we give a sufficient condition for equality. The final reduction we give (Theorem \ref{thm:addingp}) says the $p$-Kostka numbers are preserved under adding integer multiples of certain $p$-powers to the first parts of the partitions; this result is a strengthening of a reduction given by Henke in \cite{henke}.
 
 In Section \ref{sec:dec} we derive the result described above for decomposition numbers, and recall that a combinatorial description relating decomposition numbers of $S(r,r)$ and $S(rp,rp)$ corresponding to $2$-part partitions is already understood.  By example we show that this does not extend to the general case.  

In Section \ref{sec:indec} we take a different approach and classify the indecomposable Young permutation modules.  We also determine precisely when a Young permutation module has direct summands occurring in more than one block.  We summarise these results as:

\begin{thm}\label{thm:indecpartitions}
Let $r\in\N$ and let $\lambda$ be a partition of $r$, then 
\begin{enumerate}[(a)]
\item if $p$ is odd, then $M^\lambda$ is indecomposable if, and only if, one of the following holds:
\begin{itemize}
\item $p$ divides $r$ and $\lambda$ is equal to $(r)$ or $(r-1,1)$;
\item $p$ does not divide $r$ and $\lambda=(r)$.
\end{itemize}
\item if $p=2$, then $M^\lambda$ is indecomposable if, and only if, one of the following holds:
\begin{itemize}
\item $r$ is odd and $\lambda=(r)$
\item $r$ is even and $\lambda=(r)$ or $\lambda$ is one of the $n$ partitions $(r-k_i,k_i)$ where $2^{n}\leq r<2^{n+1}$, and for each $1\leq i\leq n$ the non-negative integer $k_i$ is such that $2^{i-1}\leq k_i<2^{i}$ and $k_i\equiv \frac{r-2^n}{2} \mbox{ mod $2^{i-1}$}$.
\end{itemize}
\end{enumerate}
Furthermore, if either
\begin{enumerate}[(i)]
\item $p$ is odd, but $p$ and $r$ are not both $3$, or
\item $p=2$ and $r$ is odd,
\end{enumerate} then $M^\lambda$ is decomposable if, and only if, it has summands lying outside of the principal block.  

If $p=2$ and $r$ is even then $M^\lambda$ has direct summands lying outside of the principal block if, and only if $\lambda$ has at least $3$ parts, $r\geq 6$ and $\lambda\neq (r-2,1,1)$.
\end{thm}

This has applications to the study of the endomorphism algebras of Young permutation modules.  If $\lambda$ is a partition of $r$ with at most $n$ parts then the Schur algebra $S(n,r)$ has an idempotent $e_\lambda$ (described explicitly in \cite{GreenSchurAlgebras}) for which $e_\lambda S(n,r)e_\lambda$ is isomorphic to the endomorphism algebra $\End_{k\mS_r}(M^\lambda)$.  Thus, studying the endomorphism algebras of Young permutation modules gives information on the structure of Schur algebras.  By determining the indecomposable Young permutation modules, we have determined for which $\lambda$ the algebra $e_\lambda S(n,r)e_\lambda$ is local.

 \begin{example}When the characteristic $p$ is at least $3$, it follows by Theorem \ref{thm:indecpartitions} that the partitions of $r$ indexing the indecomposable Young permutation modules are very easily understood.  For $p=2$, we give the following two examples.
 \begin{enumerate}\item We apply Theorem \ref{thm:indecpartitions} to calculate the indecomposable Young permutation modules corresponding to partitions of $126$ over a field of characteristic $2$.
In this case $n=6$ and $r-2^n=126-2^n=62$.  Hence there are precisely seven partitions $(r-k,k)$ of $126$ which correspond to indecomposable Young permutation modules, the first of these is $(126)$.  The remaining six such partitions correspond to $(r-k_i,k_i)$ where $2^{i-1}\leq k_i< 2^i$ and $31\equiv k_i\; ({\rm mod}\; 2^{i-1})$ for $1\leq i \leq n$.  It follows that $k_1=1$, $k_2=3$, $k_3=7$, $k_4=15$, $k_5=31$, and $k_6=63$. Thus, the indecomposable Young permutation modules for $k\mS_{126}$ are:
\[M^{(126)},M^{(125,1)},M^{(123,3)},M^{(119,7)},M^{(111,15)},M^{(95,31)},M^{(63,63)}.\] 
\item Let $r=2^n$ for some $n\in\N$ and let $\lambda\vdash r$.  Then $M^\lambda$ is indecomposable if, and only if, $\lambda=(r)$ or $\lambda=(r-2^i,2^i)$ for $0\leq i\leq n-1$.
\end{enumerate}\end{example}

The proof of Theorem \ref{thm:indecpartitions} is the subject of section \ref{sec:indec}.  This material can be read independently of sections \ref{sec:red} and \ref{sec:dec}, except for relying on a particular calculation - namely Example \ref{ex:r-2,1,1}.
The partitions for which $M^\lambda$ has summands in more than one block are found by analysis of the ordinary characters of the permutation modules using the Littlewood--Richardson rule.  This, along with a dimension argument is enough to classify the indecomposable Young permutation modules over fields of odd characteristic.  The remaining cases in characteristic $p=2$ are solved by analysis of the $2$-Kostka numbers for $2$-part partitions as determined by Henke (see \cite{henke}). To conclude section \ref{sec:indec} we show that the partitions indexing the indecomposable Young permutation modules in characteristic $2$ are preserved by multiplication by $2$ and also when adding certain $2$-powers.

\section{Background}
We now give a brief introduction to the background theory.  The reader is referred to \cite{jameslecture} and \cite{JamesKerber} for background on representations of symmetric groups, and \cite{Landrock} or \cite{NagaoTsushima} for representations of finite groups.  Throughout, all modules are right modules.

Fix $r\in\N$.  If $\lambda$ is a partition of $r$, then we write $\lambda\vdash r$, and  $r=\sum_{i=1}^\infty\lambda_i=\mid\lambda\mid$ is the degree of $\lambda$.  A composition of $r$ differs from a partition in that the entries need not be nonincreasing.  If $\lambda$ is a composition of $r$, which is denoted by $\lambda\models r$, then the length of $\lambda$, denoted by $l(\lambda)$ is the number of non-zero parts of $\lambda$. Following the usual conventions, $\trianglerighteq$ denotes the dominance order on partitions, and in both partitions and compositions we allow zero entries at the end.   

If $\lambda\vdash r$ and $\mu\vdash n$, then $\lambda+\mu$ is the partition of $r+n$ with $i$th part given by $\lambda_i+\mu_i$. Additionally, if $a\in\N$ and $\lambda\vdash r$, then $a\lambda$ is the partition of $ar$ obtained by multiplying every part of $\lambda$ by $a$. We say that a composition $\lambda\models r$ is divisible by $a\in\N$ if $\lambda=a\mu$ for some $\mu\models\frac{r}{a}$. The concatenation of two compositions $\lambda$ and $\mu$ is denoted by $\lambda\bullet\mu$.  A composition $\gamma$ is a refinement of $\lambda$ if there exist compositions $\gamma^{(i)}\models \lambda_i$ for each $i$, such that $\gamma=\gamma^{(1)}\bullet\dots\bullet\gamma^{(t)}$ where $t=l(\lambda)$.  Furthermore, if $p$ is prime, then a partition in which every non-zero part is a power of $p$ is called a $p$-partition.

If $\lambda=(\lambda_1,\dots,\lambda_t)\models r$, then a Young subgroup corresponding to $\lambda$, denoted $\mS_\lambda$ is any subgroup of $\mS_r$ which is conjugate to \[\mS_{\{1,\dots,\lambda_1\}}\times\mS_{\{\lambda_1+1,\dots,\lambda_1+\lambda_2\}}\times\dots\times\mS_{\{\lambda_1+\dots+\lambda_{t-1}+1,\dots,r\}}.\]

We assume a basic familiarity with the notion of $\lambda$-tabloids and $\lambda$-tableaux as defined in \cite[\S 3]{jameslecture} and follow the same notation.

Let $F$ be a field.  Then the Young permutation module $M^\lambda_F$ is the permutation $F\mS_r$-module with $F$-basis the set of all $\lambda$-tabloids.  The point stabiliser of any $\lambda$-tabloid is a Young subgroup corresponding to $\lambda$.  

An important family of modules is that of the Specht modules.  If $t$ is a $\lambda$-tableau, with column stabiliser $C_t$, and $\{t\}$ is the $\lambda$-tabloid containing $t$, then define the polytabloid $e_t:=\{t\}K_t\in M^\lambda_F$ where $K_t$ is the signed column sum $\sum_{\sigma\in C_t}\sgn(\sigma)\sigma$.  The module generated by $e_t$ is the Specht module $S^\lambda_F$.  If $F$ is a field of characteristic zero, then $\{S^\lambda_F\mid\lambda\vdash r\}$ is a complete set of mutually inequivalent irreducible modules for $F\mS_r$.

Throughout this paper, all results concerning ordinary representation theory (over fields of characteristic zero) will be stated in terms of characters.  If $F$ is a field of characteristic zero, then we denote the character of $M^\lambda_F$ by $\xi^\lambda$, the character of $S^\lambda_F$ by $\chi^\lambda$ and the usual inner product on characters by $\langle,\rangle$.    

The multiplicities $\langle\xi^\lambda,\chi^\mu\rangle$ are well known by Young's rule as follows:
\begin{thm}[Young's Rule {\cite[14.1]{jameslecture}}]\label{thm:Youngsrule}
The multiplicity $\langle\xi^\lambda,\chi^\mu\rangle$ where $\lambda,\mu\vdash r$ is equal to the number of semistandard $\mu$-tableaux of type $\lambda$.
\end{thm}
From now on, all modules are defined over a field $k$ of characteristic $p>0$.

The indecomposable direct summands of $M^\lambda$ were parameterized by James.  Specifically, in any direct sum decomposition of $M^\lambda$, there is precisely one direct summand containing $S^\lambda$ as a submodule.  Such a summand is defined up to isomorphism, and is known as the Young module corresponding to $\lambda$, denoted by $Y^\lambda$.  Moreover, every indecomposable direct summand of $M^\lambda$ is isomorphic to a Young module.  
\begin{thm}[\cite{trivsource}]\label{thm:Youngs}Let $\lambda\vdash r$, then there exists a family of self-dual indecomposable modules $\{Y^\mu\mid \mu\vdash r\}$ and non-negative integers $[M^\lambda:Y^\mu]$ for $\lambda,\mu\vdash r$ such that 
\begin{enumerate}[(a)]
\item $Y^\lambda\cong Y^\mu$ if, and only if, $\lambda=\mu$;
\item $M^\lambda\cong\bigoplus_{\mu\vdash r}[M^\lambda:Y^\mu]Y^\mu$;
\item $[M^\lambda:Y^\lambda]=1$, and if $[M^\lambda:Y^\mu]\neq 0$ then $\mu\trianglerighteq \lambda$;
\end{enumerate}
\end{thm}

In particular since $M^{(1^r)}\cong k\mS_r$, it follows that every projective indecomposable $k\mS_r$-module is isomorphic to a Young module.  A Young module $Y^\lambda$ is projective if, and only if $\lambda$ is $p$-restricted, that is, if, and only if, $\lambda_i-\lambda_{i+1}<p$ for every $i$.

The Nakayama conjecture (first proved by Brauer and Robinson in \cite{BrauerNakayama} and \cite{RobinsonNakayama}) determines the blocks of $k\mS_r$ in terms of $p$-cores of partitions. Two Specht modules $S^\lambda$ and $S^\mu$ lie in the same block of $k\mS_r$ if, and only if, $\lambda$ and $\mu$ have the same $p$-core.  Thus, to each block $B$ of $k\mS_r$ is associated a $p$-core $\tau\vdash r-pw$ where $w$ is a non-negative integer called the weight of the block.  A Specht module $S^\lambda$, irreducible character $\chi^\lambda$, or a Young module $Y^\lambda$ lie in $B$ if, and only if, the corresponding partition $\lambda$ has $p$-core $\tau$.

\section{Some reductions for $p$-Kostka numbers}\label{sec:red}
Let $G$ be a finite group.  Let $M$ be an indecomposable $kG$-module, and $H$ a subgroup of $G$, then $M$ is said to be relatively $H$-projective if there exists a $kH$-module $N$ such that $M$ is a direct summand of the induced module $N\otimes_{kH}kG=N\!\uparrow^G$.  J.A.Green (\cite{Greenvertex}) proved that every indecomposable $kG$-module $M$ defines a subgroup $Q$ of $G$ such that if $H$ is a subgroup of $G$, then $M$ is relatively $H$-projective, if, and only if, $H$ contains a conjugate of $Q$.  Such a subgroup is always a $p$-group, is called a vertex of $M$ and is defined uniquely up to conjugacy in $G$.  If $H$ is a vertex of $M$ and $N$ is an indecomposable $kH$-module such that $M$ is a direct summand of $N\!\uparrow^G$, then $N$ is called a source of $M$, and $N$ is well-defined up to conjugacy in $N_G(H)$.

We make extensive use of the Brauer construction applied to $p$-permutation modules as developed by Brou\'{e} in \cite{broue}, which we describe below.

A $kG$-module $M$ is said to be a $p$-permutation module if, for every $p$-subgroup $P$ of $G$, the restriction of $M$ to $P$ is a permutation $kP$-module.  If $M$ is a $p$-permutation module then every direct summand of $M$ is a $p$-permutation module.  An indecomposable $kG$-module $U$ is a $p$-permutation module if, and only if, $U$ has trivial source, or equivalently, $U$ is a direct summand of a permutation module.

The Brauer construction for $kG$-modules is defined as follows.  Let $M$ be a $kG$-module, and $Q$ a subgroup of $G$.  We denote by $M^Q$ the set of elements of $M$ which are fixed under the action of $Q$.  If $P$ is a subgroup of $Q$, then the relative trace map $\Tr_P^Q:M^P\to M^Q$ is defined by \[\Tr_P^Q(m)=\sum_{i}mg_i,\] where the sum runs over a set of right coset representatives for $P$ in $Q$. 

 The Brauer quotient of $M$ with respect to $Q$ is defined to be the $N_G(Q)/Q$-module  \[M(Q):=M^Q/\sum_{P<Q}\Tr_P^Q(M^P).\] 

The parametrization of indecomposable $p$-permutation $kG$-modules is as follows:
\begin{thm}[{\cite[(3.2)]{broue}}, The Brou{\'e} correspondence]\label{thm:Brouecorrespondence}
An indecomposable $p$-permutation module $M$ has vertex $P$ if, and only if, $M(P)$ is a non-zero projective $kN_G(P)/P$-module. Moreover,
\begin{enumerate}[(a)]
\item The correspondence $M\to M(P)$ induces a bijection between the isomorphism
classes of indecomposable $p$-permutation $kG$-modules with vertex $P$ and the
isomorphism classes of indecomposable projective $kN_G(P)/P$-modules.
\item Let $M$ and $U$ be $p$-permutation $kG$-modules, and assume also that $U$ is indecomposable with vertex $P$,  
then $U$ is a direct summand of $M$ if, and only if, $U(P)$ is a direct summand of
$M(P)$, and they occur with the same multiplicities.
\end{enumerate} 
\end{thm}

 Let $M$ be a permutation $kG$-module, and $P$ a $p$-subgroup of $G$.  In this case, let $B$ be a permutation basis for $M$, and let $B^P$ be the set of elements of $B$ which are fixed under the action of $P$. Then,  
\[M^P\cong \langle B^P\rangle\oplus \sum_{Q<P}\Tr_Q^P(M^Q)\;\;\mbox{as a module for $N_G(P)$}.\]  
 It follows that $M(P)$ is isomorphic to the linear span of $B^P$ as a module for $N_G(P)/P$.

We return now to the case of Young modules for symmetric groups.  Grabmeier proved in \cite{grabmeier} that every Young module has an associated Young vertex.  That is, if $r\in\N$, and $\lambda\vdash r$, then there exists a Young subgroup $\mS_\nu$ such that the Young module $Y^\lambda$ is relatively $\mS_\mu$-projective if, and only if, $\mS_\nu$ is $\mS_r$-conjugate to a subgroup of $\mS_\mu$. It follows that for every Young module $Y^\lambda$ there is a unique partition $\nu$ such that $\mS_\nu$ is a Young vertex of $Y^\lambda$.  

We recall now that a partition $\lambda=(\lambda_1,\dots,\lambda_t)$ is $p$-restricted if $\lambda_i-\lambda_{i+1}<p$ for every $i$.  With this notation, every partition $\lambda$ has a unique $p$-adic expansion $\lambda=\sum_{i}\lambda(i)p^i$ where the $\lambda(i)$ are $p$-restricted partitions.  

Let $\lambda$ be a partition of $r$, then Grabmeier describes a Young vertex of $Y^\lambda$ as follows: Let $\lambda$ have $p$-adic expansion $\lambda=\sum_i\lambda(i)p^i$, and for each $i$, let $r_i=|\lambda(i)|$.  Define $\rho$ to be the partition of $r$ with $r_i$ parts equal to $p^i$ for each $i$, then $\mS_\rho$ is a Young vertex of $Y^\lambda$. 
 Grabmeier proved (\cite[4.7]{grabmeier}) that if $\mS_\nu$ is a Young vertex of $Y^\lambda$ then a Sylow $p$-subgroup of $\mS_\nu$ is a vertex of $Y^\lambda$.

If $\nu$ is a $p$-partition, and $P_\nu$ is a Sylow $p$-subgroup of $\mS_\nu$ then $P_\nu$ and $\mS_\nu$ have the same orbits  for the natural action on $\{1,\dots,r\}$.  In particular, $P_\nu$ is the trivial subgroup if, and only if, $\nu=(1^r)$.  It follows that a Young module $Y^\lambda$ is projective if, and only if, $\lambda$ is $p$-restricted.

We will now describe the Brauer quotient of a Young module:
We denote the outer tensor product by $\boxtimes$.  That is, if $G$ and $H$ are finite groups and $M$ is a $kG$-module, $N$ a $kH$-module, then $M\boxtimes N$ is the $k(G\times H)$-module with underlying vector space $M\otimes_k N$, and action given by linearly extending  \[(m\otimes n)(g,h)=mg\otimes nh.\]

Erdmann (\cite{ErdmannYoung}) parametrized the Young modules using only the representation theory of symmetric groups, and in the process described the Brauer quotients of Young modules as follows:

\begin{thm}[\cite{ErdmannYoung}]\label{ErdmannBroue}
Let $\rho$ be a $p$-partition of $r$, and for each $i$, let $r_i$ be the number of times $p^i$ occurs as a part of $\rho$.  Let $P$ be a Sylow $p$-subgroup of $\mS_\rho$.  
\begin{enumerate}[(a)]
\item  $N_{\mS_r}(\mS_\rho)=N_{\mS_r}(P)\mS_\rho$ and if $\beta$ is the composition $(r_0,r_1,\dots)$, then \[\mS_\beta\cong N_{\mS_r}(\mS_\rho)/\mS_\rho\cong N_{\mS_r}(P)/N_{\mS_\rho}(P).\]
\item  Let $\lambda\vdash r$, then $N_{\mS_\rho}(P)/P$ acts trivially on $Y^\lambda(P)$, so we may view $Y^\lambda(P)$ as a module for $\mS_\beta$;
\item   If $\mS_\rho$ is a Young vertex of $Y^\lambda$ then $P$ is a vertex of $Y^\lambda$, and \[Y^\lambda(P)\cong Y^{\lambda(0)}\boxtimes\dots\boxtimes Y^{\lambda(t)}\;\mbox{as a $\mS_\beta$-module,}\]
where $\sum_{i=0}^t\lambda(i)p^i$ is the $p$-adic expansion of $\lambda$. 
\end{enumerate}
\end{thm}
Part (a) of the above theorem is proved by the Frattini argument, whilst part (c) is the main result.  We briefly discuss now part (b), following the approach of Erdmann: 

Let $\mu\vdash r$, with $p$-adic expansion $\sum_{i=0}^t\mu(i)p^i$.  Let $\rho$ be the $p$-partition such that $\mS_\rho$ is a Young vertex of $Y^\mu$, and $P$ a Sylow $p$-subgroup of $\mS_\rho$.  A permutation basis for $M^\lambda$ can be taken as the set of all $\lambda$-tabloids.  That is, the set of row equivalence classes of $\lambda$-tableaux.  We use the standard notation as described in \cite[\S 3]{jameslecture}.  As such, we may identify $M^\lambda(P)$ with the linear span of the $\lambda$-tabloids which are fixed by $P$. A $\lambda$-tabloid is fixed by $P$ if, and only if, each row is a union of $P$-orbits. 
  Since the $\mS_\rho$ orbits on $\{1,\dots,r\}$ coincide with the $P$-orbits, it follows that each $\lambda$-tabloid which is fixed by $P$ is also fixed by $\mS_\rho$.  Thus, $\mS_\rho$ acts trivially on $M^\lambda(P)$, and the structure of $M^\lambda(P)$ as a module for $N_{\mS_r}(P)/P$ is the same as the structure of $M^\lambda(P)$ as a module for $N_{\mS_r}(\mS_\rho)/N_{\mS_\rho}(P)\cong\mS_\beta$ (where $\beta$ is given as in Theorem \ref{ErdmannBroue}).

Throughout this section we make use of the Brou\'e correspondence to determine $p$-Kostka numbers.  Recall that by Theorem \ref{thm:Brouecorrespondence}, if $\lambda,\mu\vdash r$ and $P$ is a vertex of $Y^\mu$, then \begin{equation}[M^\lambda:Y^\mu]=[M^\lambda(P):Y^\mu(P)].\end{equation}  Thus, we must understand the Brauer quotient of the Young permutation modules with respect to Sylow $p$-subgroups of a Young subgroup $\mS_\nu$ where $\nu$ is a $p$-partition.  
 Since $\mS_\beta$ acts by permuting the $\mS_\rho$-orbits of the same length amongst themselves, $M^\lambda(P)$ is a permutation module for $\mS_\beta$.  It follows from this construction that $M^\lambda(P)$ is a direct sum of outer tensor products of Young permutation modules
\[M^{\gamma^{(0)}}\boxtimes\dots\boxtimes M^{\gamma^{(s)}}\] for some refinements $\gamma=\gamma^{(0)}\bullet\dots\bullet\gamma^{(s)}$ of $\beta$.

Assume $M^\lambda(P) \cong \bigoplus_{\gamma} M^{\gamma^{(0)}}\boxtimes\dots\boxtimes M^{\gamma^{( s)}}$, as a module for $\mS_\beta$, where the sum runs over some refinements $\gamma$ of $\beta$. From
$[M^\lambda:Y^\mu]=[M^\lambda(P):Y^\mu(P)]$, and the form for $Y^\mu(P)$ in Theorem \ref{ErdmannBroue}, follows:
\begin{equation}[M^\lambda(P):Y^\mu(P)]
=\sum_{\gamma}\prod_{i=0}^s[M^{\gamma^{(i)}}:Y^{\mu(i)}].\label{eqn:decompMY}
\end{equation}

 Thus, provided one knows the $p$-Kostka numbers corresponding to projective Young modules for all partitions of degree less than $r$, then the Brauer construction and the Brou{\'e} correspondence is sufficient to determine the non-projective $p$-Kostka numbers in degree $r$.  
\begin{remark}
In fact, Erdmann (\cite[Proposition 1]{ErdmannYoung}) describes $M^\lambda(P)$ explicitly as the direct sum \[M^\lambda(P)\cong \bigoplus_{\gamma}M^{\gamma^{(0)}}\boxtimes\dots\boxtimes M^{\gamma^{(s)}},\] where the sum is over all refinements $\gamma=\gamma^{(0)}\bullet\dots\bullet\gamma^{(s)}$ of $\beta$ where $\gamma^{(i)}$ is a composition of $r_{i}$ for each $i$, and where $\sum_{i=0}^s\gamma^{(i)}p^i=\lambda$.  Substituting this expression into (\ref{eqn:decompMY}), one obtains the Klyachko multiplicity formula.  It follows that many of the results proved in this section can be proved in a similar way using the Klyachko formula, or by considering weight spaces for simple modules for the Schur algebra and the Steinberg Tensor Product Theorem (see the proof of the Klyachko formula in \cite{MartinSchurAlgebras}).  In fact, it is possible to rephrase nearly all the proofs of this section into this language.  However, in some cases, we find it simplifies certain aspects of the combinatorics to work with the Brauer construction, and we prove these results using only the representation theory of the symmetric groups.
\end{remark}

The object of this is to provide a method for calculating $p$-Kostka numbers.  An example will be illuminating.
  Recall that every integer $r$ has a $p$-adic expansion, and we write this as $r=\sum_{i=0}^\infty \lfloor r\rfloor_i p^i$.

\begin{example}\label{ex:r-2,1,1}
We consider the case $p=2$ and want to calculate a direct sum decomposition of  $M^{(r-2,1,1)}$ for $r$ even.  We may assume $r\geq 4$. Recall that \[M^{(r-2,1,1)}\cong Y^{(r-2,1,1)}\oplus \bigoplus_{\mu\triangleright (r-2,1,1)}[M^{(r-2,1,1)}:Y^\mu]Y^\mu.\]  The dominance lattice for partitions $\mu\trianglerighteq(r-2,1,1)$ is as follows:
\[(r)\triangleright(r-1,1)\triangleright(r-2,2)\triangleright(r-2,1,1).\]
 Thus, to determine the decomposition of $M^{(r-2,1,1)}$ we need only concern ourselves with the multiplicities of the Young modules corresponding to the partitions $(r-2,2)$, $(r-1,1)$ and $(r)$.    We first calculate  $[M^{(r-2,1,1)}:Y^{(r-1,1)}]$.  The $2$-adic expansion of $(r-1,1)$ is \[(r-1,1)=(1,1)+\sum_{i=1}^t(\lfloor r-2\rfloor_{i})2^i.\]  
 
 It follows that a Young vertex of $Y^{(r-1,1)}$ is of the form $\mS_\rho$ where $\rho$ is a $2$-partition, having exactly two parts equal to $1$. 
Let $P$ be a Sylow $2$-subgroup of $\mS_\rho$, then $P$ is a vertex of $Y^{(r-1,1)}$, and the set consisting of the tabloids  \[\youngtabloid(1\dots {r-2},{r-1},r),\;\;\;\;\youngtabloid(1\dots {r-2},r,{r-1})\]
is a basis for $M^{(r-2,1,1)}(P)$ over $k$.  The structure of $M^{(r-2,1,1)}(P)$ as a $N_{\mS_r}(P)/P$-module is the same as it has as a module for $N_{\mS_r}(P)/N_{\mS_\rho}(P)$.   This group is isomorphic to $\mS_{(2,1^{t-2})}\cong C_2$.  Under this identification a generator of $C_2$ acts by permuting the entries $r-1$ and $r$ and fixing entries $1,\dots,r-2$ in the tabloids above. Thus, $M^{(r-2,1,1)}(P)\cong k\mS_{(2,1^{t-2})}\cong Y^{(1,1)}\boxtimes Y^{(1)}\boxtimes\dots\boxtimes Y^{(1)}$ as a module for $\mS_{(2,1^{t-2})}$.  It follows by the Brou\'{e} correspondence that $M^{(r-2,1,1)}$ has exactly one indecomposable summand with vertex $P$, namely $Y^{(r-1,1)}$, and $[M^{(r-2,1,1)}:Y^{(r-1,1)}]=1$. 

  Next, let $\nu$ be the $2$-partition such that $\mS_\nu$ is a Young vertex of $Y^{(r-2,2)}$.  Since all parts of $(r-2,2)$ are divisible by $2$, it follows that all parts of $\nu$ are divisible by $2$. Hence there are no $(r-2,1,1)$-tabloids which are fixed by the action of $\mS_\nu$.  It follows that if $Q$ is a Sylow $2$-subgroup of $\mS_\nu$ then $M^{(r-2,1,1)}(Q)=0$, and hence $M^{(r-2,1,1)}$ has no summands with vertex $Q$.  In particular $[M^{(r-2,1,1)}:Y^{(r-2,2)}]=0$.  An identical line of reasoning proves that $[M^{(r-2,1,1)}:Y^{(r)}]=0$.  Hence we have determined the direct sum decomposition: \[M^{(r-2,1,1)}\cong Y^{(r-2,1,1)}\oplus Y^{(r-1,1)}.\] 
\end{example}

Some of the above argument generalises as follows:

\begin{lemma}
Let $r$ be a positive integer divisible by $p$.  If $\lambda$ is a partition of $r$ not divisible by $p$, and $\mu$ is a partition of $r$ divisible by $p$, then $[M^\lambda:Y^\mu]=0$. 
\end{lemma}
\begin{proof}
Since $p$ divides every part of $\mu$, it follows that if $\rho$ is a $p$-partition of $r$ such that $\mS_\rho$ is a Young vertex of $Y^\mu$, then $p$ divides every part of $\rho$.  Let $P$ be a Sylow $p$-subgroup of $\mS_\rho$, then $P$ is a vertex of $Y^\mu$.  Since the orbits of $P$ all have length divisible by $p$, it follows that no $\lambda$-tabloid is fixed by $P$.  Thus, $M^\lambda(P)=0$, and in particular, $[M^\lambda:Y^\mu]=0$. 
\end{proof}

Next, we investigate the effect of multiplying partitions by $p$ on the $p$-Kostka numbers.
\begin{prop}\label{prop:multiplyingbyp}
Let $r\in\N$ and let $\mu,\nu\vdash r$ such that $\mS_\nu$ is a Young vertex of $Y^\mu$.  Then $\mS_{p\nu}$ is a Young vertex of $Y^{p\mu}$, and if $P_\nu$ and $P_{p\nu}$ are Sylow $p$-subgroups of $\mS_\nu$ and $\mS_{p\nu}$ respectively, then there exists a composition $\beta$ such that \[N_{\mS_r}(P_\nu)/N_{\mS_{\nu}}(P_\nu)\cong \mS_\beta\cong N_{\mS_{pr}}(P_{p\nu})/N_{\mS_{p\nu}}(P_{p\nu}).\]
Under this identification, the following hold:
\begin{enumerate}[(a)]
\item $Y^\mu(P_\nu)\cong Y^{\mu(0)}\boxtimes \dots\boxtimes Y^{\mu(s)}\cong Y^{p\mu}(P_{p\nu})$ as modules for $\mS_\beta$.
\item Let $\lambda\vdash r$, then $M^\lambda(P_\nu)\cong M^{p\lambda}(P_{p\nu})$ as modules for $\mS_\beta$.
\end{enumerate}
\end{prop}
\begin{proof}
If $\mu=\sum_{i=0}^s\mu(i)p^i$ is the $p$-adic expansion of $\mu$, then the $p$-adic expansion of $p\mu$ is $\sum_{i=0}^s\mu(i)p^{i+1}$.  Thus, if $\nu$ is the partition with $|\mu(i)|$ parts equal to $p^i$ then $\mS_\nu$ is a Young vertex of $Y^\mu$ and $\mS_{p\nu}$ is a Young vertex of $Y^{p\mu}$.  
  
Let $P_\nu$ and $P_{p\nu}$ be Sylow $p$-subgroups of $\mS_{\nu}$ and $\mS_{p\nu}$ respectively.  If $\beta:=(|\mu(0)|,|\mu(1)|,\dots,|\mu(s)|)$, then by Theorem \ref{ErdmannBroue}(a),
\begin{align*} N_{\mS_r}(P_\nu)/N_{\mS_\nu}(P_\nu)&\cong N_{\mS_r}(\mS_\nu)/\mS_\nu\cong \mS_\beta\\
\intertext{ and similarly}  N_{\mS_{pr}}(P_{p\nu})/N_{\mS_{p\nu}}(P_{p\nu})&\cong N_{\mS_{pr}}(\mS_{p\nu})/\mS_{p\nu}\cong \mS_\beta.\end{align*}
Statement $(a)$ follows directly from the expression for the $p$-adic expansion of $p\mu$, and  Theorem \ref{ErdmannBroue}.

  Recall that there is a basis $\B_\lambda$  of $M^\lambda(P_\nu)$ consisting of the $\lambda$-tabloids with rows given by unions of $\mS_\nu$-orbits.  Similarly there is a basis $\B_{p\lambda}$ of $M^{p\lambda}(P_{p\nu})$ consisting of the $p\lambda$-tabloids with rows given by unions of $\mS_{p\nu}$-orbits. Both $M^\lambda(P_\nu)$ and $M^{p\lambda}(P_{p\nu})$ are modules for $k\mS_\beta$, with the corresponding action of $\mS_\beta$ on tabloids given by extending the action of $\mS_\beta$ on the $\mS_\nu$ and $\mS_{p\nu}$ orbits respectively. 
If $l(\nu)=s$ then we label the $\mS_\nu$ orbits (respectively $\mS_{p\nu}$-orbits) $O_1,\dots,O_s$ (respectively $\bar{O}_1,\dots,\bar{O}_s$), so $O_i=\{\sum_{j=1}^i\nu_j+1,\dots,\sum_{j=1}^{i+1}\nu_{j}\}$ (respectively $\bar{O}_i=\{\sum_{j=1}^i(p\nu_i)+1,\dots,\sum_{j=1}^{i+1}(p\nu_j)\}$).  Let $x$ be a tabloid in $\B_\lambda$.  If the $i$th row of $x$ consists of $\bigcup_{j=1}^mO_{i_j}$ for each $i$, then there is a unique $p\lambda$-tabloid $\bar{x}\in\B_{p\lambda}$ with $i$th row given by $\bigcup_{j=1}^m\bar{O}_{i_j}$ for each $i$.  This defines a one-to-one correspondence between $\B_\lambda$ and $\B_{p\lambda}$.  Since $\mS_\beta$ permutes orbits of the same length amongst themselves, this correspondence induces an isomorphism $M^\lambda(P_\nu)\cong M^{p\lambda}(P_{p\nu})$ of $\mS_\beta$-modules.
\end{proof}

\begin{proof}[Proof of Theorem \ref{thm:multbyp}]
We follow the setup of Proposition \ref{prop:multiplyingbyp}.  Then \[[M^{p\lambda}:Y^{p\mu}]=[M^{p\lambda}(P_{p\nu}):Y^{p\mu}(P_{p\nu})]=[M^{\lambda}(P_\nu):Y^\mu(P_\nu)]=[M^\lambda:Y^\mu],\]
where the middle equality follows from the isomorphisms (as $k\mS_\beta$-modules) in Proposition \ref{prop:multiplyingbyp}.
\end{proof}

We compare now multiplicities of direct summands of $M^{\lambda+p^n\alpha}$ with multiplicities of direct summands of $M^\lambda$ and $M^\alpha$, for partitions $\lambda$ and $\alpha$.
\begin{prop}\label{prop:rpluspna}
Let $a,r\in\N$, and let $\mu\vdash r$, $\delta\vdash a$.  Suppose $\mu$ has $p$-adic expansion $\sum_{i=0}^s\mu(i)p^i$ and let $n>s$.  Let $\rho,\gamma$ be partitions such that $\mS_\rho,\mS_\gamma$ are Young vertices of $Y^\mu,Y^\delta$ respectively.  Assume also that $N_{\mS_r}(\mS_\rho)/\mS_\rho\cong \mS_\beta$, and $N_{\mS_a}(\mS_\gamma)/\mS_\gamma\cong \mS_\eta$ where $\beta$ and $\eta$ are given by Theorem \ref{ErdmannBroue}.
\begin{enumerate}[(a)]
\item If $\nu=\rho\bullet(p^n\gamma)$, then $\mS_\nu$ is a Young vertex of $Y^{\mu+p^n\delta}$ and \[N_{\mS_{r+ap^n}}(\mS_{\nu})/\mS_\nu\cong \mS_{\beta\bullet\eta}\cong\mS_\beta\times\mS_\eta;\]
\item Let $P_\rho$ and $P_{p^n\gamma}$ be Sylow $p$-subgroups of $\mS_\rho$ and $\mS_{p^n\gamma}$ respectively, with disjoint supports.  Then $P_\nu:=P_\rho\times P_{p^n\gamma}$ is a Sylow $p$-subgroup of $\mS_\nu$, and \[Y^{\mu+p^n\delta}(P_\nu)\cong  Y^{\mu}(P_\rho)\boxtimes Y^{p^n\delta}(P_{p^n\gamma})\;{\rm as}\;{\rm a}\;{\rm module}\;{\rm for}\; \mS_\beta\times\mS_\eta;\]
\item Let $\lambda\vdash r$ and $\alpha\vdash a$. There exists a direct summand $N$ of $M^{\lambda+p^n\alpha}(P_\nu)$ such that $N\cong M^\lambda(P_\rho)\boxtimes M^{p^n\alpha}(P_{p^n\gamma})$ as modules for $\mS_\beta\times\mS_\eta$.  

Moreover, if $p^n>\lambda_1$ then $M^{\lambda+p^n\alpha}(P_\nu)\cong M^\lambda(P_\rho)\boxtimes M^{p^n\alpha}(P_{p^n\gamma})$ as modules for $\mS_\beta\times\mS_\eta$.
\end{enumerate}
\end{prop}
\begin{proof}
If $\mu=\sum_{i=0}^s\mu(i)p^i$ is the $p$-adic expansion of $\mu$, and $\delta=\sum_{i=0}^l\delta(i)p^i$ is the $p$-adic expansion of $\delta$, then since $n>s$, it follows that \begin{equation}\label{eqn:addingdeltapn}\mu+p^n\delta=\sum_{i=0}^s\mu(i)p^i+\sum_{j=n}^{n+l}\delta(j-n)p^j\end{equation} is the $p$-adic expansion of 
$\mu+p^n\delta$.  The first part of $(a)$ now follows easily from Grabmeier's description of a Young vertex, and the second part of $(a)$ follows from Theorem \ref{ErdmannBroue}(a).

By Proposition \ref{prop:multiplyingbyp}, $P_{p^n\gamma}$ is a vertex of $Y^{p^n\delta}$.  Thus, applying Theorem \ref{ErdmannBroue} yields \begin{eqnarray}\label{eqn:mu} Y^\mu(P_\rho)&\cong &Y^{\mu(0)}\boxtimes\dots\boxtimes Y^{\mu(s)}\;\;{\rm as}\;{\rm a}\;{\rm module}\;{\rm for}\;\mS_\beta\\\label{eqn:pndelta}Y^{p^n\delta}(P_{p^n\gamma})&\cong & Y^{\delta(0)}\boxtimes\dots\boxtimes Y^{\delta(l)}\;\;{\rm as}\;{\rm a}\;{\rm module}\;{\rm for}\;\mS_\eta.
\end{eqnarray}
We have already seen that $P_\nu$ is a vertex of $Y^{\mu+p^n\delta}$, and hence by Theorem \ref{ErdmannBroue} and the $p$-adic expansion given in (\ref{eqn:addingdeltapn}), we see \[Y^{\mu+p^n\delta}(P_\nu)\cong Y^{\mu(0)}\boxtimes\dots\boxtimes Y^{\mu(s)}\boxtimes Y^{\delta(0)}\boxtimes\dots\boxtimes Y^{\delta(l)}\] as modules for $\mS_{\beta\bullet\eta}$.  Identifying $\mS_{\beta\bullet\eta}$ with $\mS_\beta\times\mS_\eta$ and noting (\ref{eqn:mu}) and (\ref{eqn:pndelta}), we obtain the isomorphism in claim $(b)$.

For each $i< p^n$, the orbits of $\mS_\nu$ of length $i$ are the same as the orbits of $\mS_\rho$ of length $i$, and for each $i\geq p^n$, the orbits of $\mS_\nu$ of length $i$ are identified with the orbits of $\mS_{p^n\gamma}$ of length $i$ (by shifting the entries by $r$).
Let $\B$ be the basis for $M^{\lambda+p^n\alpha}(P_\nu)$ consisting of $\lambda+p^n\alpha$-tabloids whose rows are unions of $\mS_\nu$-orbits.
Let $\B_\lambda$, and $\B_{p^n\alpha}$ be the corresponding tabloid bases for $M^\lambda(P_\rho)$, and $M^{p^n\alpha}(P_{p^n\gamma})$ respectively.  If $x_1,x_2$ are tabloids lying in $\B_\lambda$ and $\B_{p^n\alpha}$ respectively, then define a tabloid $x\in\B_{\lambda+p^n\alpha}$ by taking the $i$th row of $x$ to be the union of the $\mS_\nu$-orbits corresponding to the $\mS_{p^n\gamma}$-orbits and $\mS_{\rho}$-orbits occurring in the $i$th rows of $x_1$ and $x_2$ respectively.  This defines an injective map of $\B_\lambda\times\B_{p^n\alpha}$ into $\B_{\lambda+p^n\alpha}$.  It is easily seen that if $p^n>\lambda_1$, then in fact this map is a bijection.

 $\mS_\beta\times\mS_\eta$ acts by permuting the $\mS_\nu$-orbits of the same length amongst themselves, and $\mS_\beta$ and $\mS_\eta$ act on $\B_\lambda$ and $B_{p^n\alpha}$ by permuting the $\mS_\rho$ and $\mS_{p^n\gamma}$ orbits of the same length amongst themselves.   It follows that this embedding of $\B_\lambda\times\B_{p^n\alpha}$ into $\B_{\lambda+p^n\alpha}$ identifies a subset of $\B_{\lambda+p^n\alpha}$, which is closed under the action of $\mS_\beta\times\mS_\eta$ (that is, it is a union of $\mS_\beta\times\mS_\eta$-orbits). This subset therefore forms a basis for a direct summand $N$ of $M^{\lambda+p^n\alpha}(P_\nu)$.  Moreover,  $N\cong M^\lambda(P_\rho)\boxtimes M^{p^n\alpha}(P_{p^n\gamma})$ as a module for $\mS_\beta\times\mS_\eta$.  This completes the proof.\end{proof}

\begin{thm}\label{thm:plusalphaplusdelta}
Let $a,r\in\N$ and let $\lambda,\mu\vdash r$, and $\alpha,\delta\vdash a$.  Let $\mu$ have $p$-adic expansion  $\mu=\sum_{i=0}^{s}\mu(i)p^i$.  If $n>s$ then \begin{equation}[M^{\lambda+p^n\alpha}:Y^{\mu+p^n\delta}]\geq[M^\lambda:Y^\mu][M^\alpha:Y^\delta].\label{eqn:plusalphaplusdelta}\end{equation}

 Furthermore, if $p^n>\lambda_1$ then there is equality in (\ref{eqn:plusalphaplusdelta}).
\end{thm} 
\begin{proof}
Let $\rho$,$\gamma$ be partitions such that $\mS_\rho$ and $\mS_\gamma$ are Young vertices of $Y^\mu$ and $Y^\delta$ respectively.  By Proposition \ref{prop:rpluspna}, a Young vertex of $Y^{\mu+p^n\delta}$ is $\mS_\nu$ where $\nu=\rho\bullet(p^n\gamma)$.  Let $P_\nu,P_{p^n\gamma}$, and $P_\rho$ be Sylow $p$-subgroups of $\mS_\nu,\mS_{p^n\gamma}$, and $\mS_\rho$ respectively.  By Proposition \ref{prop:rpluspna}, parts $(b)$ and $(c)$, we have
\begin{equation}M^\lambda(P_\rho)\boxtimes M^{p^n\alpha}(P_{p^n\gamma}) \mid M^{\lambda+p^n\alpha}(P_\nu),\end{equation} and
\begin{equation}Y^{\mu}(P_\rho)\boxtimes Y^{p^n\delta}(P_{p^n\gamma})\cong Y^{\mu+p^n\delta}(P_\nu)\end{equation} as modules for $\mS_\beta\times\mS_\eta$.
By Theorem \ref{thm:Brouecorrespondence}, and the above discussion we have \begin{eqnarray*}[M^{\lambda+p^n\alpha}:Y^{\mu+p^n\delta}]&=&[M^{\lambda+p^n\alpha}(P_\nu):Y^{\mu+p^n\delta}(P_\nu)]\\
&\geq & [ M^{\lambda}(P_\rho)\boxtimes M^{p^n\alpha}(P_{p^n\gamma}):Y^\mu(P_\rho)\boxtimes Y^{p^n\delta}(P_{p^n\gamma})]\\
&=& [M^{\lambda}(P_\rho):Y^{\mu}(P_\rho)][M^{p^n\alpha}(P_{p^n\gamma}):Y^{p^n\delta}(P_{p^n\gamma})].
\end{eqnarray*}  
Thus, we conclude that 
\begin{equation}[M^{\lambda+p^n\alpha}:Y^{\mu+p^n\delta}]\geq [M^\lambda:Y^\mu][M^{p^n\alpha}:Y^{p^n\delta}].
\label{eqn:nearlythere}\end{equation}
Applying Theorem \ref{thm:multbyp} gives $[M^{p^n\alpha}:Y^{p^n\delta}]=[M^\alpha:Y^\delta]$ and (\ref{eqn:plusalphaplusdelta}) follows.
If $p^n>\lambda_1$, then by Proposition \ref{prop:rpluspna}(c), and an identical argument to above, we obtain equality in (\ref{eqn:plusalphaplusdelta}).  
\end{proof}

In Theorem \ref{thm:plusalphaplusdelta} we have given a sufficient condition for there to be equality in (\ref{eqn:plusalphaplusdelta}).  However, there are many examples in which equality does not hold.  For example, taking $p=2,n=1,\lambda=(3,1^3),\mu=(3,2,1),\alpha=(1,1),\delta=(2)$. Then $[M^{\alpha}:Y^\delta]=0$. Now, $\mu+p^n\delta=(7,2,1)$ has $2$-adic expansion $4(1)+(3,2,1)$, so $Y^{(7,2,1)}$ has Young vertex $\mS_{(4,1^6)}$. If $P$ is a Sylow $p$-subgroup of $\mS_{(4,1^6)}$ then $M^{(5,3,1,1)}(P)\cong M^{(3,1^3)}\boxtimes M^{(1)}$ and $Y^{(7,2,1)}(P)\cong Y^{(3,2,1)}\boxtimes Y^{(1)}$ as modules for $\mS_{(6,1)}\cong N_{\mS_{10}}(\mS_{(4,1^6)})/\mS_{(4,1^6)}$.  It follows that $[M^{(5,3,1,1)}:Y^{(7,2,1)}]=[M^{(3,1^3)}:Y^{(3,2,1)}][M^{(1)}:Y^{(1)}]$.  It is easily checked that $[M^{(3,1^3)}:Y^{(3,2,1)}]=2$ (for instance, since $(3,2,1)$ is a $2$-core, $Y^{(3,2,1)}$ is equal to the Specht module $S^{(3,2,1)}$ and is projective, so $[M^{(3,1^3)}:Y^{(3,2,1)}]=\langle\xi^{(3,1^3)}:\chi^{(3,2,1)}\rangle$ which can be calculated via the Littlewood-Richardson rule).  Therefore
\[[M^{\lambda+p^n\alpha}:Y^{\mu+p^n\delta}]=2>0=[M^\lambda:Y^\mu][M^\alpha:Y^\delta].\]

In Theorem \ref{thm:plusalphaplusdelta} we focused on the case of adding $p$-power multiples of partitions.  In the special case when $\alpha=\delta=(a)$ we will now say much more in the following theorem which strengthens \cite[Corollary 4.2]{HenkeKoenig} and \cite[Corollary 6.2]{henke}:

\begin{thm}\label{stronger}\label{thm:addingp}
Let $\lambda\vdash r$ such that $\lambda_2<p^n$ then \[[M^{\lambda+(ap^n)}:Y^{\mu+(ap^n)}]=[M^\lambda:Y^\mu]\] for every $a\in\N$ and every $\mu\vdash r$.
\end{thm}
\begin{proof}
Let $\sum_{i\geq 0}\mu(i)p^i$ be the $p$-adic expansion of $\mu$ and define $\underline\mu=\sum_{i=0}^{n-1}\mu(i)p^i, \hat\mu=\sum_{i\geq n}\mu(i)p^i$.  Then $\mbar$ and $\mhat$ are both partitions and $\mu=\mbar+\mhat$.  For all $i$ let $r_i=\mid\mu(i)\mid$.  We also define $\underline\rho$ to be the partition of $\mid\mbar\mid$ with 
$r_i$ parts equal to $p^i$ for $i=0,\dots,n-1$  and $\hat\rho$ to be the partition of $\mid \mhat\mid$ with 
$r_i$ parts equal to $p^i$ for each $i\geq n$,

so that if $\rho=\hat\rho\bullet\underline\rho$ then $\mS_\rho$ is a Young vertex of $Y^\mu$.  Notice also that $\mS_{\hat\rho}$, and $\mS_{\underline{\rho}}$ are Young vertices for $Y^{\mhat}$ and $Y^{\mbar}$ respectively. It is clear that $\mS_{\rho}\cong\mS_{\hat\rho}\times\mS_{\underline\rho}$, and we will from here on view this isomorphism as an identification.  Now, if  $\hat{Q}$ and $\underline{Q}$ are Sylow $p$-subgroups of $\mS_{\hat\rho},\mS_{\underline\rho}$ respectively, then $Q=\hat{Q}\times\underline{Q}$ is a Sylow $p$-subgroup of $\mS_\rho$, so $\hat{Q},\underline{Q},Q$ are vertices of $Y^{\mhat},Y^{\mbar}$ and $Y^\mu$ respectively.  Setting $\underline\beta=(r_0,\dots,r_{n-1})$
and $\hat{\beta}=(r_{n},r_{n+1},\dots)$ and defining $\beta=\underline\beta\bullet\hat\beta$ we recall that $N_{\mS_{r}}(\mS_{\rho})/\mS_{\rho}\cong \mS_{\beta}\cong \mS_{\underline\beta}\times\mS_{\hat\beta}$.

We now consider the multiplicity $[M^\lambda:Y^\mu]=[M^\lambda(Q):Y^\mu(Q)]$.  Recall that $M^\lambda(Q)$ is isomorphic to the linear span of all $\lambda$-tabloids with rows unions of $\mS_{\rho}$-orbits, viewed as a module for $\mS_\beta$ where the action is given by permuting the orbits of the same length amongst themselves.  
Now, if $x_1$ is a $\lambda$-tabloid with rows all unions of $\mS_{\rho}$-orbits then since $p^n>\lambda_2$, all orbits of $\mS_{\rho}$ of size at least $p^n$ lie in the first row of $x_1$.  The union of all such $\mS_{\rho}$-orbits has size $\mid\mhat\mid$.  Thus, if $\mid\mhat\mid>\lambda_1$ then $M^\lambda(Q)=0$ and
$[M^\lambda:Y^\mu]=[M^\lambda(Q):Y^\mu(Q)]=0.$  

Otherwise, $\mid\mhat\mid\leq\lambda_1$ and we may define $\eta$ to be the composition obtained from $\lambda$ by subtracting $\mid\mhat\mid$ from $\lambda_1$.  Define $\mathcal{T}$ to be the set of all tabloids of shape $\eta$, with rows unions of all $\mS_\rho$-orbits of length less than $p^n$.  It follows that there is a bijection between $\mathcal{T}$ and the set of all $\lambda$-tabloids with rows unions of $\mS_\rho$-orbits, given by adjoining the $\mS_\rho$-orbits of size at least $p^n$ to the first row of the elements of $\mathcal{T}$.

Under the isomorphism $\mS_{\beta}\cong \mS_{\underline{\beta}}\times\mS_{\hat{\beta}}$ the subgroup $\mS_{\underline{\beta}}$ acts by permuting the      $\mS_\rho$-orbits of size less than $p^n$ of the same length amongst themselves, and $\mS_{\hat\beta}$ acts by permuting the orbits of size at least $p^n$ of the same length amongst themselves.  It follows that $M^\lambda(Q)\cong \langle\mathcal{T}\rangle\boxtimes M^{(\mid\mhat\mid)}(\hat{Q})$ as a module for $\mS_{\underline\beta}\times\mS_{\hat\beta}$.

Now, by Theorem \ref{ErdmannBroue} and the above identifications, $Y^\mu(Q)\cong Y^{\mbar}(\underline{Q})\boxtimes Y^{\mhat}(\hat{Q})$ as modules for $\mS_{\underline\beta}\times\mS_{\hat\beta}$. Hence \begin{eqnarray*}[M^\lambda:Y^\mu]&=&[\langle\mathcal{T}\rangle:Y^{\mbar}(\underline{Q})][M^{(\mid\mhat\mid)}(\hat{Q}):Y^{\mhat}(\hat{Q})].\\
&=&[\langle\mathcal{T}\rangle:Y^{\mbar}(\underline{Q})][M^{(\mid\mhat\mid)}:Y^{\mhat}].\end{eqnarray*} 
Now, $\mu+(ap^n)=\mbar+(\mhat+(ap^n))$, and we write $\sum_{i\geq n}\nu(i)p^i$ for the $p$-adic expansion of $\mhat+(ap^n)$.  Then $\sum_{i=0}^{n-1}\mu(i)p^i+\sum_{i\geq n}\nu(i)p^i$ is the $p$-adic expansion of $\mu+(ap^n)$.  Let $\gamma$ be the partition with $\mid\nu(j)\mid$ parts equal to $p^j$ for every $j\geq n$.  Then $\mS_\gamma$ is a Young vertex of $Y^{\mhat+(ap^n)}$, and $\mS_{\gamma\bullet\underline\rho}$ is a Young vertex of $Y^{\mu+(ap^n)}$.  Thus, if $\hat{P}$ is a Sylow $p$-subgroup of $\mS_{\gamma}$ then $\hat{P}\times \underline{Q}=P$ is a Sylow $p$-subgroup of $\mS_{\gamma\bullet\underline\rho}$, and hence is a vertex of $Y^{\mu+(ap^n)}$.  Also, define $\tau$ to be a composition with parts $\mid\nu(i)\mid$ for all $i$ so that $N_{\mS_{\mid\mhat\mid+ap^n}}(\mS_{\gamma})/\mS_{\gamma}\cong \mS_\tau$.  Then $N_{\mS_{r+ap^n}}(\mS_{\gamma\bullet\underline\rho})/\mS_{\gamma\bullet\underline\rho}\cong \mS_{\underline{\beta}\bullet\tau}\cong \mS_{\underline{\beta}}\times\mS_\tau$.

Now, $M^{\lambda+(ap^n)}(P)$ is isomorphic to the linear span of the $(\lambda+(ap^n))$-tabloids with rows unions of $\mS_{\gamma\bullet\underline\rho}$-orbits.   As before, if $x_2$ is such a tabloid, then the orbits of $\mS_{\gamma\bullet\underline\rho}$ of size at least $p^n$ must all lie in the first row of $x_2$, since the second part of $\lambda+(ap^n)$ is $\lambda_2$. If $\mid\mhat\mid>\lambda_1$ then $\mid\mhat+(ap^n)\mid>\lambda_1+ap^n$ and hence $M^{\lambda+(ap^n)}(P)=0$, so in particular $[M^{\lambda+(ap^n)}:Y^{\mu+(ap^n)}]=0=[M^\lambda:Y^\mu]$.  Thus, we may assume that $\mid\mhat\mid\leq\lambda_1$.  So, noting that if we subtract $\mid\mhat+ap^n\mid$ from the first part of $\lambda+(ap^n)$ we obtain $\eta$, it follows as before that the $(\lambda+(ap^n))$-tabloids with rows unions of $\mS_{\gamma\bullet\underline\rho}$-orbits are in one-to-one correspondence with the elements of $\mathcal{T}$.  In this way we obtain an isomorphism 
$M^{\lambda+(ap^n)}(P)\cong \langle\mathcal{T}\rangle\boxtimes M^{(\mid\mhat\mid+ap^n)}(\hat{P})$ as modules for $\mS_{\underline\beta}\times\mS_\tau$.  From the above construction and Theorem \ref{ErdmannBroue} follows that $\hat{P}$ is a vertex of $Y^{\mhat+(ap^n)}$ and $Y^{\mu+(ap^n)}(P)\cong Y^{\mbar}(\underline{Q})\boxtimes Y^{\mhat+(ap^n)}(\hat{P})$ as modules for $\mS_{\underline\beta}\times\mS_\tau$.  Thus,
\begin{eqnarray*}
[M^{\lambda+(ap^n)}:Y^{\mu+(ap^n)}]&=&[\langle\mathcal{T}\rangle:Y^{\mbar}(\underline{Q})][M^{(\mid\mhat\mid+ap^n)}(\hat{P}):Y^{\mhat+(ap^n)}(\hat{P})]\\
&=&[\langle\mathcal{T}\rangle:Y^{\mbar}(\underline{Q})][M^{(\mid\mhat\mid+ap^n)}:Y^{\mhat+(ap^n)}]
\end{eqnarray*}
It remains only to note that \[[M^{(\mid\mhat\mid+ap^n)}:Y^{\mhat+(ap^n)}]=[M^{(\mid\mhat\mid)}:Y^{\mhat}]=\left\{\begin{array}{ll}1&\mbox{ if $\mhat=(\mid\mhat\mid)$}\\0&\mbox{otherwise.}\end{array}\right.\] \end{proof}

\begin{remark}
It is worth noting here that the sufficient condition for equality in Theorem \ref{thm:plusalphaplusdelta} is not necessary.  For example, if $\lambda,\mu\vdash r$ with $\lambda_2<p^n$ and such that $\mu$ has $p$-adic expansion of the form $\sum_{i=0}^s\mu(i)p^i$ for $s<n$, then by Theorem \ref{stronger} it follows that there is equality in (\ref{eqn:plusalphaplusdelta}) whenever $\alpha=\delta=(a)$ for any $a\in\N$.  However, here we have no condition on $\lambda_1$.
\end{remark}

\section{A consequence for decomposition numbers}\label{sec:dec}

For this section, let $(K,R,k)$ be a $p$-modular system.  As was first shown by Scott (\cite{ScottModularPerms}), for any indecomposable $k\mS_r$-module $U$ with trivial source, there exists an indecomposable $R\mS_r$-module $\hat{U}$ with trivial source such that $\hat{U}\otimes_Rk\cong U$.  Moreover, $\hat{U}$ is unique up to isomorphism.  Thus, we associate to $U$ a well-defined ordinary character $\ch(U)$, defined as the character of $\hat{U}\otimes_RK$.   

The ordinary character of a transitive permutation module is the usual corresponding permutation character, so from Theorem \ref{thm:Youngs} follows 
\begin{equation}\xi^\lambda=\ch(Y^\lambda)+\sum_{\mu\triangleright\lambda}[M^\lambda:Y^\mu]\ch(Y^\mu),\label{ordinaryxi}\end{equation} for each $\lambda\vdash r$.  Furthermore, from Young's rule (Theorem \ref{thm:Youngsrule}) it follows that $\langle\xi^\lambda,\chi^\lambda\rangle=1$ and $\langle\xi^\lambda,\chi^\mu\rangle\neq 0$ only if $\mu\trianglerighteq\lambda$.  Thus, from equation (\ref{ordinaryxi}) it follows that $\langle\ch(Y^\lambda),\chi^\lambda\rangle=1$ and $\langle\ch(Y^\lambda),\chi^\mu\rangle\neq 0$ only if $\mu\trianglerighteq\lambda$.

It is the ordinary characters of the Young modules which relate the decomposition numbers of the Schur algebras and the $p$-Kostka numbers.  For background on the Schur algebra, we refer the reader to  \cite{GreenSchurAlgebras} and \cite{MartinSchurAlgebras}.  We adopt the usual notation and write $d_{\mu\lambda}=(\nabla(\mu):L(\lambda))$ for the decomposition numbers of the Schur algebra $S(r,r)$ (defined over an infinite field of characteristic $p$), where $\lambda,\mu\vdash r$.  Then \cite[4.6.4]{MartinSchurAlgebras} we have \begin{equation}\ch(Y^\lambda)=\sum_{\mu\vdash r}(\nabla(\mu):L(\lambda))\chi^\mu.\label{decordchar}\end{equation}
Now, by Young's rule, the multiplicities $\langle\xi^\lambda,\chi^\mu\rangle$ are known and if one knows the decomposition numbers of $S(r,r)$ then by the uni-triangular nature of the decomposition matrix, and (\ref{ordinaryxi}), one may calculate the $p$-Kostka numbers $[M^\lambda:Y^\mu]$ for all $\lambda,\mu\vdash r$.  Similarly, if one knows the $p$-Kostka numbers, then by Young's rule and (\ref{ordinaryxi}), one may calculate the multiplicities $\langle\ch(Y^\lambda),\chi^\mu\rangle=(\nabla(\mu):L(\lambda))$ for all $\lambda,\mu\vdash r$.  We can now state the main result of this section:
\begin{thm}\label{thm:decomp}
If one knows the decomposition numbers of the principal block of $S(rp,rp)$, that is, $(\nabla(\mu):L(\lambda))$ for all partitions $\lambda,\mu\vdash rp$ with empty $p$-core, then there is an algorithm to calculate the decomposition numbers of $S(r,r)$.
\end{thm}
\begin{proof}
If $\lambda\vdash r$ then $p\lambda\vdash pr$ has an empty $p$-core and hence $Y^{p\lambda}$ lies in the principal block $B_0$ of $k\mS_{rp}$. If $\gamma$ is any partition of $rp$ lying in $B_0$ then by Young's rule we may calculate the multiplicities of the irreducible constituents of the principal block part of $\xi^\gamma$, which we denote as $(\xi^{\gamma})_{B_0}=\ch((M^{\gamma})_{B_0})$.  By equation \ref{decordchar}, we may assume the ordinary characters of the Young modules lying in $B_0$ are known, and hence, since the decomposition matrix of $S(rp,rp)$ is uni-triangular, we may calculate the $p$-Kostka numbers $[M^{p\lambda}:Y^\mu]$ where $\lambda$ runs over the partitions of $r$ and $\mu$ runs over the partitions of $rp$ with empty $p$-core.  In particular, we have calculated $[M^{p\lambda}:Y^{p\mu}]$ for all $\lambda,\mu\vdash r$.  Now by Theorem \ref{thm:multbyp} we have calculated the $p$-Kostka numbers $[M^\lambda:Y^\mu]$ for all $\lambda,\mu\vdash r$ and by the discussion preceding this theorem, 
we may calculate the decomposition numbers $(\nabla(\mu):L(\lambda))$ for all $\lambda,\mu\vdash r$.
\end{proof}

A question immediately raised by Theorem \ref{thm:decomp} is to determine how one might combinatorially describe the connection between the decomposition numbers of $S(r,r)$ and those of the principal block of $S(rp,rp)$.  Henke described the decomposition numbers corresponding to two-part partitions in \cite{HenkeCartan}.  In particular, from Lemma 2.4 and Proposition 2.2 of \cite{HenkeCartan} follows that if $\lambda,\mu$ are two-part partitions of $r$ then $(\nabla(p\mu):L(p\lambda))=(\nabla(\mu):L(\lambda))$. We now give some examples to show that this does not hold in the general case.

It can be checked (for example by using the methods described in this paper), that over a field of characteristic $3$, we have the following ordinary characters:
\begin{eqnarray*}
\ch Y^{(3^3)}&=& \chi^{(3^3)}+\chi^{(4,3,2)}+\chi^{(5,4)}+\chi^{(6,2,1)}+2\chi^{(6,3)}+\chi^{(7,1^2)}+\chi^{(8,1)},\\
\ch Y^{(1^3)}&=&\chi^{(1^3)}+\chi^{(2,1)}.
\end{eqnarray*}
In particular, when $p=3$, it follows that $(\nabla(2,1):L(1,1,1))=1$ whilst $(\nabla(6,3):L(3,3,3))=2$.

Over a field of characteristic $2$ we have 
\begin{eqnarray*}
\ch Y^{(4,2^2)}&=&\chi^{(4,2^2)}+\chi^{(4,3,1)}+\chi^{(4^2)}+2\chi^{(5,3)}+\chi^{(6,1^2)}+2\chi^{(6,2)}+\chi^{(7,1)}.\\
\ch Y^{(2,1^2)}&=&\chi^{(2,1^2)}+\chi^{(2^2)}+\chi^{(3,1)}.
\end{eqnarray*}
Thus, when $p=2$ it follows that $(\nabla(3,1):L(2,1^2))=1$ whilst $(\nabla(6,2):L(4,2^2))=2.$

 \section{The indecomposable Young permutation modules}\label{sec:indec}
In this section, we proceed to a proof of Theorem \ref{thm:indecpartitions}. 
The main method of proof is to demonstrate that a Young permutation module $M^\lambda$ has composition factors lying in at least two different $p$-blocks of $\mS_r$.  For $p\geq 3$, in combination with a simple dimension argument, this is sufficient (Proposition \ref{oddprimesclassification}), and similarly for the case where $p=2$ and $r$ is odd.  For the situation where $p=2$ and $r$ is even, we apply a similar method, aside from a few exceptional cases which are dealt with separately, to reduce the problem of determining the indecomposable Young permutation modules to the case where $\lambda\vdash r$ has at most $2$ parts.  In this case, the $p$-Kostka numbers have been determined by Henke in \cite{henke}, and an analysis of these multiplicities (Proposition \ref{thm:indecperm}) yields the final classification.  

We recall the following well-known lemma:
\begin{lemma}\label{lem:triv}
Let $G$ be a finite group and $M$ a transitive permutation $FG$-module, where $F$ is some field. Then $M$ has a unique trivial submodule $C$.  Moreover, $C$ is a direct summand of $M$ if, and only if, the characteristic of $F$ does not divide $\dim_FM$.  
\end{lemma}

In particular, when $F$ has characteristic zero, and $M$ is the permutation module on the cosets of a Young subgroup $\mS_\lambda$, Lemma \ref{lem:triv} simply states that $\langle\xi^{\lambda},\chi^{(r)}\rangle=1$ for every $\lambda\vdash r$.  This is a fact we will make repeated use of.

We note a useful first case that follows easily from Lemma \ref{lem:triv}.
\begin{lemma}\label{lem:n-1}
$M^{(r-1,1)}$ is indecomposable if, and only if, $p$ divides $r$.
\end{lemma}
\begin{proof}  
By Theorem \ref{thm:Youngs} and Lemma \ref{lem:triv}, we have $M^{(r-1,1)}\cong Y^{(r-1,1)}\oplus \kappa Y^{(r)}$ where $\kappa$ is either $0$ or $1$.  By Lemma \ref{lem:triv} the result follows, since ${\rm dim}_k M^{(r-1,1)}=r$.
\end{proof}

We recall the induction product on modules (and characters) for symmetric groups.  That is, if $M$ is an $\mS_m$-module, and $N$ is an $\mS_n$-module, then the induction product is defined by $M_\bullet N=(M\otimes N)_{\mS_m\times \mS_n}\!\uparrow^{\mS_{m+n}}$.   
The induction product is both commutative and associative, and the decomposition of the induction product of two ordinary irreducible characters is given by the Littlewood--Richardson rule (see for example \cite[2.8.13 and 2.8.14]{JamesKerber}).

In the following, if $\varphi$ and $\psi$ are characters of $\mS_r$ over some field $F$ of characteristic zero, then we write $\varphi\leq \psi$ to mean that  $\langle\varphi,\chi^\lambda\rangle\leq \langle\psi,\chi^\lambda\rangle$ for all $\lambda\vdash r$.  

We require the following facts:

Let $\lambda=(\lambda_1,\dots,\lambda_t)$ be a partition of $r$ with $t\geq 2$.  Then $\lambda$ is the concatenation $\lambda=\gamma\bullet\delta$ of two partitions $\gamma=(\lambda_1,\dots,\lambda_i)$ and $\delta=(\lambda_{i+1},\dots,\lambda_t)$.  It follows that 
\begin{equation}
 \xi^\lambda={\xi^\gamma}_\bullet\xi^\delta,
\end{equation}
 and in particular, if $\chi_1\leq\xi^\gamma$ and $\chi_2\leq\xi^\delta$, then ${\chi_1}_\bullet\chi_2\leq\xi^\lambda$.
 
We recall also \cite[Example 14.4]{jameslecture} that if $d\leq \frac{r}{2}$, then 
\begin{equation}\label{eqn:2parts}\xi^{(r-d,d)}=\sum_{i=0}^d\chi^{(r-i,i)}.\end{equation}

 The key result is the following lemma:
\begin{lemma}\label{lem:chi}
Let $\lambda\vdash r$ then:
\begin{enumerate}[(a)]
\item if $l(\lambda)\geq 2$ then $\langle\xi^\lambda,\chi^{(r-1,1)}\rangle>0$;
\item if $r\geq 4$ and $l(\lambda)\geq 2$ with $\lambda\neq (r-1,1)$ then $\langle\xi^\lambda,\chi^{(r-2,2)}\rangle>0$;
\item if $r\geq 5$ and $l(\lambda)\geq 3$ with $\lambda\neq (r-2,1,1)$ then $\langle\xi^{\lambda},\chi^{(r-3,2,1)}\rangle> 0$.
\end{enumerate}
\end{lemma}
\begin{proof}
The case $l(\lambda)=2$ in $(a)$ and $(b)$ follow from (\ref{eqn:2parts}).  Thus, we may assume that $\lambda=(\lambda_1,\dots,\lambda_t)$ and $t=l(\lambda)\geq 3$.  We will argue by inclusion of Young subgroups.  It is clear that there exists an inclusion of Young subgroups $\mS_{\lambda}\leq \mS_{(r-d,d)}$ for some $d\geq 1$ and if $r\geq 4$ then there exists an inclusion of Young subgroups $\mS_\lambda\leq \mS_{(r-d,d)}$ for some $d\geq 2$.  Now, in both cases, the trivial character $1_{\mS_{(r-d,d)}}$ is a constituent of $1_{\mS_{\lambda}}\!\uparrow^{\mS_{(r-d,d)}}$.  Thus, $\xi^{(r-d,d)}=1_{\mS_{(r-d,d)}}\!\uparrow^{\mS_r}\leq 1_{\mS_{\lambda}}\!\uparrow^{\mS_r}=\xi^\lambda$.  Parts $(a)$ and $(b)$ now follow from (\ref{eqn:2parts}).

For part $(c)$: From parts $(a)$ and $(b)$ and Lemma \ref{lem:triv} follows that $\xi^{(s-2,2)}$ is a constituent of $\xi^{(\lambda_1,\dots,\lambda_{t-1})}$ where $s=\lambda_1+\dots+\lambda_{t-1}$.  Thus ${\xi^{(s-2,2)}}_\bullet\xi^{(\lambda_t)}\leq{\xi^{(\lambda_1,\dots,\lambda_{t-1})}}_\bullet\xi^{(\lambda_t)}=\xi^\lambda$.  From the commutativity and associativity of the induction product follows
\begin{equation*}
{\xi^{(s-2,2)}}_\bullet\xi^{(\lambda_t)}={\xi^{(s-2,\lambda_t)}}_\bullet\xi^{(2)}.
\end{equation*}

Now, from (\ref{eqn:2parts}) follows that $\xi^{(r-3,1)}\leq\xi^{(s-2,\lambda_t)}$.  Thus,
\begin{equation*}\xi^{(r-3,2,1)}={\xi^{(r-3,1)}}_\bullet{\xi^{(2)}}\leq {\xi^{(s-2,\lambda_t)}}_\bullet\xi^{(2)}\leq \xi^\lambda\end{equation*} 
Part $(c)$ follows by Young's rule.
\end{proof}

Recall that every positive integer $r$ has a $p$-adic expansion, and we write this as $r=\sum_{i=0}^\infty \lfloor r\rfloor_{i} p^i$.  So in particular $\lfloor r\rfloor_{0}\in\{0,1,\dots p-1\}$ and $r\equiv \lfloor r\rfloor_{0} \ ({\rm mod}\ p)$.  We note here that the $p$-core of $(r)$ is $(\lfloor r\rfloor_{0})$ and hence the principal $p$-block of $\mS_r$ is indexed by $p$-core $(\lfloor r\rfloor_{0})$.

\begin{lemma}\label{lem:12}\label{lem:2}
Let $r\in\mathbb{N}$ and let $p$ be prime. Then:
\begin{enumerate}[(a)]
\item $\chi^{(r-1,1)}$ lies in the principal $p$-block of $\mS_r$ if, and only if, $p$ divides $r$; and
\item if $p$ is odd, then $\chi^{(r-2,2)}$ lies in the principal $p$-block of $\mS_r$ if, and only if, $p$ divides $r-1$.
\end{enumerate}
\end{lemma}

\begin{proof}
This follows from calculating the relevant $p$-cores as follows:
 The $p$-core of $(r-1,1)$ is \begin{displaymath} \left\{\begin{array}{ll}
 \emptyset & \mathrm{if}\; p\; \mathrm{divides}\;  r,\\
(\lfloor r-1\rfloor_{0},1) & \mathrm{if}\; \lfloor r\rfloor_{0}\neq 0\; \mathrm{and}\; \lfloor r-1\rfloor_{0}\neq 0,\\
(p,1) & \mathrm{if}\; \lfloor r\rfloor_{0}\neq 0\; \mathrm{and}\; \lfloor r-1\rfloor_{0}=0.\\ \end{array}\right.\end{displaymath}

If $p$ is odd, then the $p$-core of $(r-2,2)$ is \begin{displaymath} \left\{\begin{array}{ll}
(p+1,2) & \mathrm{if}\; \lfloor r-2\rfloor_{0}=1,\\
(1,1) & \mathrm{if}\; \lfloor r-2\rfloor_{0}=0,\\
(1) & \mathrm{if}\; \lfloor r-2\rfloor_{0}=p-1,\\
(\lfloor r-2\rfloor_{0},2) & \mathrm{otherwise.}\; \end{array}\right.\end{displaymath}
\end{proof}
We can now classify the indecomposable Young permutation modules over fields of odd characteristic.  
\begin{prop}\label{oddprimesclassification}
Let $p$ be prime.  Let $r\in\mathbb{N}$, and $\lambda\vdash r$.  Then the following hold:
\begin{enumerate}[(a)]
\item If $p$ does not divide $r$ then $M^\lambda$ is indecomposable if, and only if, $\lambda=(r)$.
\item  If $p$ is odd, and divides $r$, then $M^\lambda$ is indecomposable if, and only if, $\lambda=(r-1,1)$, or $\lambda=(r)$.
\end{enumerate}
Furthermore, if either
\begin{enumerate}[(i)]
\item $p$ is odd, but $p$ and $r$ are not both $3$, or
\item $p=2$ and $r$ is odd,
\end{enumerate} then $M^\lambda$ is decomposable if, and only if, it has summands lying outside of the principal block.  
\end{prop}
\begin{proof}
Suppose $p$ does not divide $r$. We may assume that $l(\lambda)\geq 2$. Then by Lemma \ref{lem:triv} and Lemma \ref{lem:chi}, both $\chi^{(r-1,1)}$, and $\chi^{(r)}$ have non-zero multiplicity in $\xi^\lambda$.  By Lemma \ref{lem:12} these two characters lie in different $p$-blocks.  It follows that $M^\lambda$ has composition factors lying in different blocks and hence is decomposable.

We now consider the case of $p$ being odd and dividing $r$. By Lemma \ref{lem:n-1}, the permutation module $M^{(r-1,1)}$ is indecomposable, and certainly $M^{(r)}$ is the trivial module and hence is indecomposable.  For the case $r=3$, and $p=3$ this leaves only $M^{(1,1,1)}\cong k\mS_3$ which is certainly decomposable (it decomposes into projective indecomposable summands $Y^{(2,1)}$ and $Y^{(1,1,1)}$ corresponding to the two non-isomorphic simple $k\mS_3$-modules).  The result therefore holds for $r=3$.  If $r\geq 4$ and $\lambda\neq (r),(r-1,1)$, then by Lemma \ref{lem:triv} and Lemma \ref{lem:chi} the characters $\chi^{(r-2,2)}$ and $\chi^{(r)}$ have non-zero multiplicity in $\xi^\lambda$.  By Lemma \ref{lem:2}, the characters $\chi^{(r-2,2)}$ and $\chi^{(r)}$ lie in different $p$-blocks.  It follows that $M^\lambda$ is decomposable.  This completes the proof.
\end{proof}

Thus, for odd primes we have completed the classification. We are left only to complete the less straightforward case of $p=2$ and $r$ even. We first reduce the problem to $2$-part partitions by similar methods to those used above.

\begin{prop}\label{lem:even}
Let $p=2$ and $r\in\mathbb{N}$ be even.  Let $\lambda\vdash r$ such that $l(\lambda)\geq 3$. Then 
\begin{enumerate}[(a)]
\item  $M^\lambda$ is decomposable;
\item  if $r\geq 6$  and $\lambda\neq (r-2,1,1)$ then $M^\lambda$ has at least one indecomposable direct summand which does not lie in the principal $2$-block of $\mS_r$.
\end{enumerate}
\end{prop}
\begin{proof}
First we address the case $r=4$.  This is a special case since $\mS_4$ has just one $2$-block.  We have two partitions to consider, namely $(2,1^2)$, and $(1^4)$.
We note here again that $M^{(1^4)}\cong k\mS_4$ and hence, since $k\mS_4$ has 2 simple modules, this is decomposable with two non-isomorphic direct summands, the projective indecomposable $k\mS_4$-modules. The remaining partition $(2,1^2)$ is a particular case of $(r-2,1^2)$ for $r\geq 4$.  We have shown in Example \ref{ex:r-2,1,1} that for $r$ even and $r\geq 4$, we have \[M^{(r-2,1,1)}\cong Y^{(r-2,1,1)}\oplus Y^{(r-1,1)}.\] In particular $M^{(r-2,1,1)}$ is decomposable.

Assume now that $r\geq 6$ and $r$ is even.  Since $l(\lambda)\geq 3$, it follows from Lemma \ref{lem:chi} and Lemma \ref{lem:triv} that for $\lambda\neq (r-2,1,1)$ we have $\langle\xi^{\lambda},\chi^{(r-3,2,1)}\rangle\neq 0$ and $\langle\xi^{\lambda},\chi^{(r)}\rangle =1$ .  The $2$-core of $(r-3,2,1)$ is $(3,2,1)$, so $(r-3,2,1)$ does not lie in the principal block of $k\mS_r$.  Hence $M^\lambda$ has direct summands occurring in different 2-blocks of $\mS_r$.
\end{proof}

Proposition \ref{lem:even} yields that in characteristic $2$, if $r\in\mathbb{N}$ is even and $\lambda\vdash r$ with $M^\lambda$ indecomposable then $\lambda=(r)$ or $\lambda$ is a 2-part partition.  Henke determined the $p$-Kostka numbers for $2$-part partitions in \cite{henke}.
\begin{thm}[{\cite[Theorem 3.3]{henke}}]\label{thm:2partKostka}
Let $r\in\mathbb{N}$.  If $0\leq s\leq j\leq r/2$, then $Y^{(r-s,s)}$ is a direct summand of $M^{(r-j,j)}$ if, and only if, $\binom{r-2s}{j-s}\not\equiv 0\; (mod\; p)$. 

 Moreover, if $[M^{(r-j,j)}:Y^{(r-s,s)}]\neq 0$ then $[M^{(r-j,j)}:Y^{(r-s,s)}]= 1$.
\end{thm} 

In order to efficiently calculate binomial coefficients modulo a prime we use Lucas' Theorem which says the following:
\[\binom{r-2s}{j-s}\equiv \prod_i\binom{\lfloor r-2s\rfloor_{i}}{\lfloor j-s\rfloor_{i}}\; \mathrm{mod}\; p.\]

For the remainder of this section, we fix $p=2$.

For $m\in\mathbb{N}$ define $\nu(m)$ such that $2^{\nu(m)}$ is the highest power of $2$ which divides $m$. So in particular $\nu(m)=\min\{i\in\mathbb{N}\mid\lfloor m\rfloor_{i}\neq 0\}$.
Before proceeding to the final part of the classification, we give a preliminary example:
\begin{example}\label{ex:jj}
Let $j\in\N$.  If $0\leq s<j$ then since $\lfloor 2(j-s)\rfloor_{\nu(j-s)}=0$, it follows by Lucas' Theorem that $\binom{2(j-s)}{j-s}$ is even.  Thus, $M^{(j,j)}$ is indecomposable.
\end{example}

\begin{prop}\label{prop:indecsdone}
Let $j,r\in\mathbb{N}$ such that $r$ is even and $r\geq 2j>0$.  Then $M^{(r-j,j)}$ is indecomposable if, and only if, $2^{n_j}$ divides $r-2j$ where $n_j:=\min\{\alpha\in\mathbb{N}\mid j<2^\alpha\}$.

Furthermore, if $2^{n_j}$ does not divide $r-2j$ then \[[M^{(r-j,j)}:Y^{(r-j+2^{\nu(r-2j)},j-2^{\nu(r-2j)})}]=1.\]
\end{prop}
\begin{proof}
 We write $r=2j+\beta$.  We assume first that $2^{n_j}$ divides $\beta$.  If $\beta=0$ then by Example \ref{ex:jj}, the permutation module $M^{(j,j)}$ is indecomposable.  Thus, we may assume that $\beta>0$.  
 
Since $2^{n_j}$ divides $\beta$, it follows that $\nu(\beta)\geq n_j$. Let $s\in\N$ be such that $0\leq s\leq j-1$.  Then we have 
\[j-s\leq j<2^{n_j}\leq 2^{\nu(\beta)}\leq \beta,\] and hence $\nu(j-s)< \nu(\beta)$.

Now $r-2s=2(j-s)+\beta$, and hence $\lfloor r-2s\rfloor_{\nu(j-s)}=\lfloor 2(j-s)\rfloor_{\nu(j-s)}=0.$
On the other hand, $\lfloor j-s\rfloor_{\nu(j-s)}=1$, and hence 
\[\binom{\lfloor r-2s\rfloor_{\nu(j-s)}}{\lfloor j-s\rfloor_{\nu(j-s)}}=0.\]
By Lucas' Theorem, the binomial coefficient $\binom{r-2s}{j-s}\equiv 0\;{\rm mod}\;2$, and hence by Theorem \ref{thm:2partKostka}, it follows that $[M^{(r-j,j)}:Y^{(r-s,s)}]=0$.  This holds for all $0\leq s\leq j-1$.  Therefore if $\nu(\beta)\geq n_j$ then $M^{(r-j,j)}$ is indecomposable. 
 
Suppose now that $2^{n_j}$ does not divide $\beta$.  Then $\beta>0$ and $\nu(\beta)<n_j$.  Define $s:=j-2^{\nu(\beta)}$, then $s\geq 0$.  By Lucas' Theorem, the binomial coefficient \[\binom{r-2s}{j-s}=\binom{\beta+2^{\nu(\beta) +1}}{2^{\nu(\beta)}}\] is odd, so by  Theorem \ref{thm:2partKostka}, the multiplicity $[M^{(r-j,j)}:Y^{(r-s,s)}]$ is equal to $1$, and, in particular $M^{(r-j,j)}$ is decomposable. This completes the proof.
\end{proof}

In Proposition \ref{prop:indecsdone}, we have described all indecomposable Young permutation modules for the case $p=2$ and $r$ even, giving infinite families of two part partitions for which they correspond.  We now give a (perhaps) more satisfying description, applying Proposition \ref{prop:indecsdone} to completely describe for a given even $r\in\mathbb{N}$, all the partitions $\lambda$ of $r$ such that $M^\lambda$ is indecomposable:

\begin{prop}\label{thm:indecperm}
Let $r,n\in\N$ such that $r$ is even, and $2^n\leq r<2^{n+1}$. 
  Then  there are precisely $n+1$ partitions $\lambda$ of $r$ with the property that $M^\lambda$ is indecomposable.  These partitions are precisely the partition $(r)$ and the partitions given for each $1\leq i\leq n$ by $(r-k_i,k_i)$, where each $k_i$ is the unique integer such that $2^{i-1}\leq k_i<2^i$ and $k_i\equiv \frac{r-2^n}{2}\; {\rm mod}\; 2^{i-1}$.
\end{prop}
\begin{proof}
$M^{(r)}\cong k$ is an indecomposable Young permutation module.  
Let $j\in\N$ be such that $0<j\leq \frac{r}{2}$ and define $n_j$ as in Proposition \ref{prop:indecsdone}, so $2^{n_j-1}\leq j<2^{n_j}$.  If $M^{(r-j,j)}$ is indecomposable then $r-2j=2^{n_j}\alpha$ for some $\alpha\in\mathbb{N}$, by Proposition \ref{prop:indecsdone}.  Since $n_j\leq n$, it follows that $j\equiv \frac{r-2^n}{2}\; {\rm mod}\; 2^{n_j-1}$. 
 Conversely, let $1\leq i\leq n$ then there exists a unique $j$ such that $n_j=i$ and $j\equiv \frac{r-2^n}{2}\; {\rm mod}\; 2^{n_j-1}$.  Then $r=2j+2^{n_j}\gamma$ for some $\gamma\in\N$.  Hence, by Proposition \ref{prop:indecsdone}, the Young permutation module $M^{(r-j,j)}$ is indecomposable.  
\end{proof}

We note here that $k_1=1$, so we have again shown that $M^{(r-1,1)}$ is indecomposable whenever $2$ divides $r$.  It also follows that $k_n=r/2$, corresponding to Example \ref{ex:jj}.

There are some properties of the families of partitions labelling indecomposable permutation modules which are of interest.  We outline some of these now.

By Theorem \ref{thm:multbyp}, it follows that if $\lambda,\mu\vdash r$, then $[M^\lambda:Y^\mu]=[M^{2\lambda}:Y^{2\mu}]$, and hence, if $M^\lambda$ is indecomposable then $[M^{2\lambda}:Y^{2\mu}]=0$ for every partition $\mu$ of $r$ such that $\mu\neq \lambda$.  To determine whether $M^{2\lambda}$ is indecomposable, would require to prove $[M^{2\lambda}:Y^\nu]=0$ whenever $\nu\triangleright 2\lambda$ and at least one part of $\nu$ is odd.  In the following proposition we apply Theorem \ref{thm:indecperm} to the case when $r$ is even, to show that $M^{2\lambda}$ must be indecomposable, and all except one of the indecomposable Young permutation modules for $k\mS_{2r}$ occur in this way.

\begin{prop}\label{prop:multby2indecs}
Let $r\in\N$ be even and let $\lambda\vdash r$.  If $M^\lambda$ is indecomposable then $M^{2\lambda}$ is indecomposable.

Furthermore, there is a bijection between $\{\lambda\vdash r\mid M^\lambda\;{\it is}\;{\it indecomposable}\}$  and $\{\mu\vdash 2r\mid M^\mu\;{\it is}\;{\it indecomposable}\}\backslash \{(2r-1,1)\}$ given by multiplication by $2$.
\end{prop}
\begin{proof}
Let $n$ be such that $2^n\leq r<2^{n+1}$ and define $\beta:=r-2^n$.  By Proposition \ref{thm:indecperm} there are precisely $n+1$  partitions $\lambda$ of $r$ such that $M^\lambda$ is indecomposable.  These partitions $\lambda$  are given by $(r)$ and $(r-k_i,k_i)$ where $2^{i-1}\leq k_i< 2^i$ and $k_i\equiv \frac{\beta}{2}\;{\rm mod}\;2^{i-1}$ for $1\leq i\leq n$.  Then $2^i\leq 2k_i<2^{i+1}$ and $2k_i\equiv \beta\;{\rm mod}\; 2^i$.  Since multiplication by $2$ produces the $n+1$ partitions $(2r)$ and $(2r-2k_i,2k_i)$ for $1\leq i\leq n$, an application of Proposition \ref{thm:indecperm} to $2r$ gives the result.  
\end{proof}

\begin{prop}\label{prop:referenced}
Let $r\in\N$ be even and $n$ be such that $2^{n}\leq r< 2^{n+1}$.  Let $0\leq j\leq r/2$ and let $n\leq k$, then $M^{(r-j,j)}$ is indecomposable if, and only if, $M^{(r+2^{k}-j,j)}$ is indecomposable.
\end{prop}
\begin{proof}
Write $r=2^n+\beta$ with $0\leq \beta<2^n$.  Then $r+2^{k}=2^{k}+2^n+\beta$.  An application of Proposition \ref{thm:indecperm} then yields the result. \end{proof}

\bibliographystyle{plain}

\end{document}